\documentclass[11pt]{amsart}
\usepackage{amsfonts}
\usepackage{amssymb}
\usepackage{amsmath,amsthm}
\usepackage{amscd}

\newtheorem{theorem}{Theorem}

\newtheorem{definition}[theorem]{Definition}

\newtheorem{lemma}[theorem]{Lemma}

\newtheorem{proposition}[theorem]{Proposition}

\input{tcilatex}

\begin{document}
\author{Stefano Marmi}
\address{Scuola Normale Superiore, Piazza dei Cavalieri, 7, 56126 Pisa, Italy%
}
\email{s.marmi@sns.it}
\author{Piergiulio Tempesta}
\address{Scuola Normale Superiore, Centro di Ricerca Matematica "Ennio De
Giorgi" Piazza dei Cavalieri, 3, 56126 Pisa, Italy and Departamento de
Fisica Teorica II, Facultad de Fisicas, Universidad Complutense, 28040
Madrid, Spain.}
\email{p.tempesta@sns.it}
\title[Polylogs and Lipschitz--type formulae]{Polylogarithms, hyperfunctions
and \\
generalized Lipschitz summation formulae \ \ \ \ \ \ \ \ \ \ \ \ \ \ \ \ \ \ 
}
\maketitle

\begin{abstract}
A generalization of the classical Lipschitz summation formula is proposed.
It involves new polylogarithmic rational functions constructed via the
Fourier expansion of certain sequences of Bernoulli--type polynomials.
Related families of one--dimensional hyperfunctions are also constructed.
\end{abstract}

\tableofcontents

\section{Introduction}

The purpose of this paper is to provide a natural setting which allows to
generalize the Lipschitz summation formula to negative powers.

We recall that the classical Lipschitz summation formula gives the Fourier
series expansion of the periodic analytic function obtained by summation
over integer translates of the power $z^{-k}$, where $z\in \mathfrak{H}$ \
(the complex upper half plane) and $k$ is a positive integer.

By introducing suitable generalizations of the Lipschitz formula, we will
show how to construct new classes of hyperfunctions. Indeed, each of these
generalized formulae will provide a hyperfunctional equation involving a
specific two--variable polylogarithm series.

A simple introduction to the theory of hyperfunctions is sketched in Section
2, where the interested reader is referred to the classical works \cite{SKK}%
, \cite{Horm}. Hyperfunctional cohomology groups have also been related to
automorphic forms and period functions \cite{Bruggeman}.

The main results of this paper are essentially two. The first result is a
generalization of the classical Lipschitz summation formula. The second one
is the connection we establish between a new class of polylogarithms, the
theory of hyperfunctions in one variable and that of polynomial structures
of Appell type.

Precisely, our construction can be summarized as follows. We will define
suitable sequences of Appell polynomials of Bernoulli type: indeed, they
share with the Bernoulli polynomials several arithmetic properties,
including certain famous congruences (see the Appendix). Their periodic
versions provide primitives of the periodic delta function. To each sequence
we will associate a natural extension of the notion of polylogarithmic
function defined in the unit disc. We call this extension a \textit{delta
rational function}. It is a two--variable Dirichlet series, extending to the
whole Riemann sphere as a meromorphic function. The case of the classical
Bernoulli polynomials corresponds to the standard polylogarithmic function,
and it is treated thoroughly. The new Appell polynomial sequences we
construct, $\left\{ P_{n}\left( x\right) \right\} _{n\in \mathbb{N}}$ and $%
\left\{ Q_{n}\left( x\right) \right\} _{n\in \mathbb{N}}$ (depending on
certain parity properties) and the generalized polylogarithms $\delta
_{n}\left( q\right) $ proposed here satisfy interesting hyperfunctional
equations. These equations in turn correspond to generalized Lipschitz
summation formulae. For instance, in the case of polynomials $\left\{
P_{n}\left( x\right) \right\} _{n\in \mathbb{N}}$ it reads

\begin{equation}
\sum_{k\in \mathbb{Z}}\varphi _{\overline{P}_{n}}\left( \tau +k\right)
=2i\left( 2\pi i\right) ^{-n}\left\{ 
\begin{array}{c}
\Delta _{-n}\left( q\right) \text{ \ \ \ \ \ \ \ \ \ \ \ \ \ \ \ \ \ \ }if%
\text{ }\left\vert q\right\vert <1,\,\ i.e.\text{ }\Im \tau >0 \\ 
\\ 
\left( -1\right) ^{n-1}\Delta _{-n}\left( q^{-1}\right) \text{ \ \ }if\text{ 
}\left\vert q\right\vert >1,\,\ i.e.\text{ }\Im \tau <0%
\end{array}%
\right. \text{,}  \label{I.1}
\end{equation}%
where $\overline{P}_{n}$ are the hyperfunctions associated to the polynomial 
$P_{n}$ and $\varphi _{\overline{P}_{n}}$ is the function in $O^{1}\left( 
\overline{\mathbb{C}}\backslash \left[ 0,1\right] \right) $ representing $%
\overline{P}_{n}$.

In the last part of the paper, the relation between the theory of formal
groups and that of hyperfunctions is clarified. We will show that the Appell
structures introduced in the previous construction can be viewed as
polynomial realizations of certain formal group laws, related with the
universal Lazard formal group.

The future research plans include an extension of the proposed construction
to the multidimensional case (see also \cite{PD} for a different
generalization), as well as possible applications of the proposed Lipschitz
formulae to the study of Eisenstein series and periods of modular forms.

\textbf{Acknowledgements}. One of us (P. T.) is grateful to U. Zannier for a
very useful discussion.

\vspace{2mm}

\section{Hyperfunctions: basic preliminaries}

\vspace{1mm}

In this section, we will provide a brief and self--consistent introduction
to the theory of hyperfunctions of a single variable, following closely \cite%
{SKK}, Chapter IX of \cite{Horm}, \cite{Kaneko} and \cite{Marmi}. For a more
extensive treatment and further details, the reader is invited to consult
these books, as well as \cite{Pham} for interesting applications.

Let us denote by $\mathcal{O}$ the sheaf of holomorphic functions on
\thinspace $\mathbb{C}$, and denote by $\mathbb{C}^{+}$ and $\mathbb{C}^{-}$
the upper and lower half planes of $\mathbb{C}$.

\begin{definition}
The space of hyperfunctions $\mathcal{B}$ on the real line $\mathbb{R}$ is%
\begin{equation}
\mathcal{B}\left( \mathbb{R}\right) :=H_{\mathbb{R}}^{1}\left( \mathbb{C},%
\mathcal{O}\right) \text{,}  \label{H1}
\end{equation}

i.e. is the first sheaf cohomology group on $\mathbb{R}$.
\end{definition}

Now, since $\mathbb{C}^{+}$ $\cup $ $\mathbb{C}^{-}=\mathbb{C}$ $\backslash $
$\mathbb{R}$, we have the following decomposition:%
\begin{equation}
H_{\mathbb{R}}^{1}\left( \mathbb{C},\mathcal{O}\right) =\left[ H^{0}\left( 
\mathbb{C}^{+},\mathcal{O}\right) \oplus H^{0}\left( \mathbb{C}^{-},\mathcal{%
O}\right) \right] \text{ }/\text{ }H^{0}\left( \mathbb{C},\mathcal{O}\right) 
\text{.}  \label{H2}
\end{equation}%
In other words, since $H^{0}\left( \mathcal{O}\right) $ represents nothing
but the global sections of the sheaf (i.e. holomorphic functions), a
hyperfunction can be thought as a pair of holomorphic functions on the upper
and lower half planes respectively, modulo an entire function. This
geometric definition immediately generalizes to the case of hyperfunctions
in several variables. If $\Omega \subset \mathbb{R}$ is an open set, and $U$
an arbitrary complex neighborhood of $U$, then clearly%
\begin{equation}
\mathcal{B}\left( \Omega \right) =H_{\Omega }^{1}\left( U,\mathcal{O}\right) 
\text{.}  \label{H3}
\end{equation}%
Still denoting by $\mathcal{O}\left( U\right) $ the space of holomorphic
functions in $U$, another equivalent definition is%
\begin{equation}
\mathcal{B}\left( \Omega \right) :=\lim_{\longrightarrow }\text{ }_{U\text{ }%
\supset \text{ }\Omega }\text{ \ }\mathcal{O}\left( U\text{ }\backslash 
\text{ }\Omega \right) \text{ }/\text{ }\mathcal{O}\left( U\right)
\label{H4}
\end{equation}

where the inductive limit with respect to the family of complex
neighborhoods $U$ $\supset $ $\Omega $ is considered. A hyperfunction $%
f\left( x\right) $ is therefore an equivalence class $\left[ F\left(
z\right) \right] $, whose representative is \thinspace $F\left( z\right) \in 
\mathcal{O}\left( U\text{ }\backslash \text{ }\Omega \right) $. The
representative $F\left( z\right) $ is said to be a defining function of $%
f\left( x\right) $. Since%
\begin{equation*}
U\supset \Omega =U_{+}\cup U_{-}\text{, \qquad }U_{\pm }=U\cap \left\{ \Im
z\lessgtr 0\right\} \text{,}
\end{equation*}%
often the following boundary--value representation is used:%
\begin{equation*}
f\left( x\right) =F_{+}\left( x+i0\right) -F_{-}\left( x-i0\right) \text{,}
\end{equation*}%
with $F_{\pm }\left( z\right) =\left. F\left( z\right) \right\vert _{U_{\pm
}}$.

Alternatively, one can construct a theory of hyperfunctions based on
analytic functionals. Let $K\subset \mathbb{C}$ \ be a non empty compact
set, and denote by $A$ the space of entire analytic functions in $\mathbb{C}$%
.

\begin{definition}
The space $A^{\prime }\left( K\right) $ of the analytic functionals carried
by $K$ is the space of linear forms u acting on A such that for every
neighborhood V of K there is a constant $C_{V}>0$ such that%
\begin{equation}
\left\vert u\left( \varphi \right) \right\vert \leq C_{V}\text{ }%
\sup_{V}\left\vert \varphi \right\vert \text{, \qquad }\forall \varphi \in A%
\text{.}  \label{H5}
\end{equation}
\end{definition}

Observe that $A^{\prime }\left( K\right) $ is a Fr\'{e}chet space, since a
seminorm is associated to each neighborhood $V$ of $K$. One can define%
\begin{equation}
\mathcal{B}\left( \Omega \right) :=A^{\prime }\left( \overline{\Omega }%
\right) /A^{\prime }\left( \partial \Omega \right) \text{.}  \label{H6}
\end{equation}

It is interesting to notice that the space of hyperfunctions $\psi $ $\in
B\left( \Omega \right) $ \ with compact support $K\subset \Omega $ can be
identified with analytic functionals in $A^{\prime }\left( \mathbb{R}\right) 
$ with support $K$. Indeed, an analytic functional $u$ on $\cup
_{i=1}^{r}K_{i}$ can always be decomposed into a sum $u=u_{1}+...+u_{r}$,
with each of the functionals $u_{j}\in A^{\prime }\left( K_{j}\right) $.
Consequently, since supp $\psi \subset K\cup \partial \Omega $, the
contribution of $\psi $ on $\partial \Omega $ can be factored out and $\psi $
is identified with an uniquely defined functional with support in $\Omega $.

Therefore, we can also think of $A^{\prime }\left( K\right) $ as the space
of hyperfunctions with support in $K$. \ The link between the two approaches
to the theory of hyperfunctions is now provided by the following lemma. Let $%
\mathcal{O}^{1}\mathcal{(}\overline{\mathbb{C}}$ $\backslash $ $K\mathcal{)}$
denotes the space of holomorphic functions on $\overline{\mathbb{C}}$ $%
\backslash $ $K$ and vanishing at infinity.

\begin{lemma}
\label{th1} The spaces $A^{\prime }\left( K\right) $ and $\mathcal{O}^{1}%
\mathcal{(}\overline{\mathbb{C}}$ $\backslash $ $K\mathcal{)}$ are
canonically isomorphic. To each $u\in A^{\prime }\left( K\right) $ it
corresponds a function $\varphi \in \mathcal{O}^{1}\mathcal{(}\overline{%
\mathbb{C}}$ $\backslash $ $K\mathcal{)}$ given by%
\begin{equation*}
\varphi \left( z\right) =u\left( c_{z}\right) \text{, \ \ }\forall z\in 
\mathbb{C}\text{ }\backslash \text{ }K\text{,}
\end{equation*}%
where $c_{z}\left( x\right) =\frac{1}{\pi }\frac{1}{x-z}$. Conversely, to
each $\varphi \in \mathcal{O}^{1}\mathcal{(}\overline{\mathbb{C}}$ $%
\backslash $ $K\mathcal{)}$ it corresponds the hyperfunction%
\begin{equation}
u\left( \psi \right) =\frac{i}{2\pi }\int_{\gamma }\varphi \left( z\right)
\psi \left( z\right) dz\text{, \ \ }\forall \psi \in A\text{,}  \label{H7}
\end{equation}%
where $\gamma $ is any piecewise $\mathcal{C}^{1}$path winding around K in
the positive direction.
\end{lemma}

For future purposes, we also briefly describe \textit{periodic hyperfunctions%
}. Let $\mathbb{T}^{1}=\mathbb{R}$ $\backslash $ $\mathbb{Z}\subset \mathbb{C%
}$ $\backslash $ $\mathbb{Z}$. A hyperfunction on $\mathbb{T}^{1}$ is a
linear functional $\Psi $ on the space $\mathcal{O}\left( \mathbb{T}%
^{1}\right) $ of functions analytic in a complex neighborhood of $\mathbb{T}%
^{1}$ such that for all neighborhood $V$ of $\mathbb{T}^{1}$ there exists a
constant $C_{V}>0$ such that%
\begin{equation}
\left\vert \Psi \left( \Phi \right) \right\vert \leq C_{V}\text{ }%
\sup_{V}\left\vert \Phi \right\vert ,\text{ \qquad }\forall \Phi \in 
\mathcal{O(}V\mathcal{)}\text{.}  \label{H8}
\end{equation}

We will denote by $A^{\prime }\left( \mathbb{T}^{1}\right) $ the Fr\'{e}chet
space of hyperfunctions with support in $\mathbb{T}^{1}$. Let $\mathcal{O}%
_{\Sigma }$ denote the complex vector space of holomorphic functions $\Phi :%
\mathbb{C}$ $\backslash $ $\mathbb{R\rightarrow C}$, 1--periodic, bounded at 
$\pm $ $i\infty $ and such that $\Phi \left( \pm i\infty \right) :=\lim_{\Im
z\rightarrow \pm \infty }\Phi \left( z\right) $ exist and verify $\Phi
\left( +i\infty \right) =-\Phi \left( -i\infty \right) $.

\begin{lemma}
The spaces $A^{\prime }\left( \mathbb{T}^{1}\right) $ and $\mathcal{O}%
_{\Sigma }$ are canonically isomorphic: to each $\Psi $ $\in $ $A^{\prime
}\left( \mathbb{T}^{1}\right) $ it corresponds a function $\Phi \in \mathcal{%
O}_{\Sigma }$ given by%
\begin{equation*}
\Phi \left( z\right) =\Psi \left( C_{z}\right) ,\text{ \ }\forall z\in 
\mathbb{C}\text{ }\backslash \text{ }K\text{,}
\end{equation*}

where $C_{z}\left( x\right) =\cot \pi \left( x-z\right) $. Conversely, to
each $\Phi \in \mathcal{O}_{\Sigma }$ it corresponds the hyperfunction%
\begin{equation}
\Psi \left( \Xi \right) =\frac{i}{2}\int_{\Gamma }\Phi \left( z\right) \Xi
\left( z\right) dz\text{, \ \ }\forall \text{ }\Xi \in A^{\prime }\left( 
\mathbb{T}^{1}\right) \text{,}  \label{H9}
\end{equation}
where $\Gamma $ is any piecewise $\mathcal{C}^{1}$ path winding around a
closed interval $I\subset \mathbb{R}$ of length 1 in the positive direction.
\end{lemma}

Given a compactly supported hyperfunction, making infinitely many copies of
it translating its support leads to a periodic hyperfunction. This point of
view is systematically exploited in \cite{Marmi}, from which the following
commutative diagram is taken:

\begin{equation*}
\quad \ \ \ \ \ \ \ A^{\prime }\left( \left[ 0,1\right] \right) \text{ \
\quad\ \ \ \ }\longrightarrow \text{ }\ \ \ \text{\ \ \quad }O^{1}\left( 
\overline{\mathbb{C}}\text{ }\backslash \text{ }\left[ 0,1\right] \right)
\end{equation*}

\begin{equation*}
\Sigma _{\mathbb{Z}}\downarrow \text{ \ \ \ \ \ \ }\ \ \ \ \ \ \ \ \ \ \ \ 
\text{\ \ \ \ \ \ \ \ \ \ \ \ }\downarrow \Sigma _{\mathbb{Z}}
\end{equation*}

\begin{equation}
A^{\prime }\left( \mathbb{T}^{1}\right) \text{ \ }\ \text{\ \ \ \ \ }\ \ \
\longrightarrow \text{ }\ \ \ \ \ \ \ \ \text{\ \ \ \ \ }\mathcal{O}_{\Sigma
}  \label{H10}
\end{equation}

The horizontal lines are the above mentioned isomorphisms, and $\sum_{%
\mathbb{Z}}$ denotes the summation over integer translates.

\vspace{2mm}

\section{Delta rational functions}

\subsection{Delta rational functions and polylogarithms}

\vspace{2mm}

In this section we introduce the notion of delta rational functions, which
represent a natural generalization of the notion of polylogarithms \cite%
{Lewin}, \cite{Zagier}. These functions are connected to the classical
Bernoulli polynomials via the Fourier expansion formula (\ref{1.3}) and
enter a generalized Lipschitz summation formula providing new classes of
hyperfunctions.

\begin{definition}
\bigskip The delta rational function is the series expressed for any $n\in 
\mathbb{Z}$, and $q\in \mathbb{D}$ by%
\begin{equation}
\delta _{n}\left( q\right) =\sum_{k=1}^{\infty }k^{n}q^{k}\text{.}
\label{1.1}
\end{equation}
\end{definition}

Observe that, if $n\geq 0$, $\delta _{n}$ extends to the whole Riemann
sphere as a rational function of degree $n+1$ with just a simple pole of
order $n+1$ at $q=1$. If $n\leq -1$ then $\delta _{n}$ coincides with the
classical polylogarithm series of order $-n$ (actually $\delta _{-1}\left(
q\right) =-\log \left( 1-q\right) $) and extends to the whole $\mathbb{C}$ $%
\backslash \lbrack 1,+\infty )$ and as a multiplicative function to the
whole $\mathbb{C}$ $\backslash \{0,1,\infty \}$. Indeed, since%
\begin{equation*}
q\partial _{q}\delta _{n}\left( q\right) =\delta _{n+1}\left( q\right) ,
\end{equation*}%
one can define the analytic continuation of $\delta _{n}$. For instance, the
continuation of $\delta _{-2}$ is obtained by means of the integral formula%
\begin{equation*}
\delta _{-2}\left( q\right) =-\int_{0}^{q}\frac{\log \left( 1-t\right) }{t}%
dt=\int_{0}^{q}\left( \int_{0}^{t}\frac{d\zeta }{1-\zeta }\right) \frac{dt}{t%
}\text{.}
\end{equation*}%
Note that $[1,+\infty )$ is a branch cut. For all $n\in \mathbb{Z}$ and $%
q\in \mathbb{D}$, one has%
\begin{equation}
\delta _{n}\left( q^{k}\right) =k^{-1-n}\sum_{\Lambda ^{k}=1}\delta
_{n}\left( \Lambda q\right) \text{,}  \label{1.2}
\end{equation}%
where $\Lambda $ denotes a $k$--th root of unity. The equality extends to
the closed disk if $n\leq -2$. One can directly prove the following result.

\begin{lemma}
The fundamental inversion equation holds%
\begin{equation}
\delta _{n}\left( q\right) +\left( -1\right) ^{n}\delta _{n}\left(
q^{-1}\right) =\left\{ 
\begin{array}{c}
0\qquad \qquad \qquad \qquad \qquad \ \ \ if\text{ }n\geq 1 \\ 
\\ 
-\frac{\left( 2\pi i\right) ^{-n}}{\left( -n\right) !}B_{-n}\left( \frac{%
\log q}{2\pi i}\right) \text{ \qquad }if\text{ }n\leq 0%
\end{array}%
\right. \text{,}  \label{1.3}
\end{equation}

for all $q\neq 1$ if $n\geq 0$, and $q$ in $%
%TCIMACRO{\U{2102} }%
%BeginExpansion
\mathbb{C}
%EndExpansion
\backslash \left[ 0,+\infty \right] $ if $n\leq 0$. Here $B_{k}$ is the $k$%
--th Bernoulli polynomial.
\end{lemma}

\subsection{Periodic hyperfunctions}

\vspace{1mm}

The inversion relation (\ref{1.3}) has a beautiful interpretation in terms
of hyperfunctions, as we will show below.

\begin{definition}
The periodic Bernoulli functions and distributions are expressed by%
\begin{equation}
\widetilde{B}_{n}\left( x\right) =\left\{ 
\begin{array}{c}
-\frac{B_{n}\left( x-\left[ x\right] \right) }{n!}\qquad if\text{ }n\geq 1
\\ 
-1+\delta _{\mathbb{T}}\left( x\right) \text{ \ \ \ }if\text{ }n=0 \\ 
\delta _{\mathbb{T}}^{\left( -n\right) }\left( x\right) \qquad if\text{ }n<0%
\end{array}%
\right. \text{,}  \label{2.1}
\end{equation}%
where $\left[ x\right] $ is the integer part of $x$, $\delta _{\mathbb{T}}$
is the periodic delta distribution and $\delta _{\mathbb{T}}^{\left(
k\right) }$ its derivative of order $k\geq 0$.
\end{definition}

Explicitly,%
\begin{equation}
\delta _{\mathbb{T}}\left( x\right) =\sum_{k=-\infty }^{+\infty }e^{2\pi ikx}%
\text{.}  \label{2.2}
\end{equation}%
From the well--known property of the periodic Bernoulli functions%
\begin{equation}
\frac{d^{n}}{dx^{n}}\left( B_{n}\left( x-\left[ x\right] \right) \right)
=n!\left( 1-\delta _{\mathbb{T}}\right) \text{,}  \label{2.3}
\end{equation}%
we see that we can consider the Bernoulli functions as primitives of the
periodic delta function. It is immediate to check that the following
hyperfunctional equation holds.

\begin{proposition}
\bigskip Let $q^{\pm }=e^{2\pi i\left( x\pm i0\right) },$ $x\in \mathbb{R}$.
For all $n\in \mathbb{Z}$ we have%
\begin{equation}
\delta _{n}\left( q^{+}\right) +\left( -1\right) ^{n}\delta _{n}\left(
\left( q^{-}\right) ^{-1}\right) =\left( 2\pi i\right) ^{-n}\widetilde{B}%
_{-n}\left( x\right) \text{.}  \label{2.4}
\end{equation}
\end{proposition}

Note that on both sides one can apply derivatives $q\partial _{q}$ and $%
\left( 2\pi i\right) ^{-1}\partial _{x}$.

\subsection{Generalized Lipschitz summation formula}

\vspace{1mm}

As is well known, the classical Lipschitz formula is a consequence of the
Poisson summation formula, relating sums over integers of pairs of Fourier
transforms. It states that%
\begin{equation}
\sum_{n\in \mathbb{Z}}\frac{1}{\left( n+z\right) ^{k}}=\frac{\left( -2\pi
i\right) ^{k}}{\left( k-1\right) !}\sum_{r=1}^{\infty }r^{k-1}e^{2\pi
irz},\qquad z\in \mathbb{H},\qquad k\in \mathbb{Z}_{\geq 2}  \label{3.0}
\end{equation}%
(for a proof of (\ref{3.0}), see, for instance, D. Zagier, Chapter 4 in \cite%
{WMLI}).

In this section, we will give an hyperfunctional generalization of (\ref{3.0}%
). We make use of the fact that periodic hyperfunctions with compact support
are the result of a summation over integer translates of hyperfunctions with
support on $\left[ 0,1\right] $, according to the commutative diagram (\ref%
{H10}). By applying the results described in the previous section, we get a
functional version of the Lipschitz formula.

\begin{theorem}
F$or$ $all$ $n\in \mathbb{Z}$ we have%
\begin{equation}
\sum_{k\in \mathbb{Z}}\varphi _{\overline{B}_{n}}\left( \tau +k\right)
=2i\left( 2\pi i\right) ^{-n}\left\{ 
\begin{array}{c}
\delta _{-n}\left( q\right) \text{ \ \ \ \ \ \ \ \ \ \ \ \ \ \ \ \ \ \ }if%
\text{ }\left\vert q\right\vert <1,\,\ i.e.\text{ }\Im \tau >0 \\ 
\\ 
\left( -1\right) ^{n-1}\delta _{-n}\left( q^{-1}\right) \text{ \ \ }if\text{ 
}\left\vert q\right\vert >1,\,\ i.e.\text{ }\Im \tau <0%
\end{array}%
\right. \text{,}  \label{3.1}
\end{equation}%
whereas the usual Lipschitz formula corresponds to $n\leq -1$. Here $%
\overline{B}_{n}$ is the restriction to $\left[ 0,1\right] $ of $\widetilde{%
B_{n}}$ and $\varphi _{\overline{B}_{n}}$ is the function in $O^{1}\left( 
\overline{\mathbb{C}}\backslash \left[ 0,1\right] \right) $ which represents
the hyperfunction $\overline{B}_{n}$%
\begin{equation}
\varphi _{\overline{B}_{n}}\left( \tau \right) =\left\langle \overline{B}%
_{n},\frac{1}{\pi }\frac{1}{x-\tau }\right\rangle _{\left[ 0,1\right] }.
\label{3.2}
\end{equation}
\end{theorem}

\begin{proof}
It is enough to remark that an inverse of the operator $\sum_{\mathbb{Z}}$
on the hyperfunctions is just the restriction of a periodic hyperfunction to
the interval $\left[ 0,1\right] $. Thus the r.h.s. of (\ref{2.4}) reads%
\begin{equation}
\sum_{k\in \mathbb{Z}}\left( 2\pi i\right) ^{-n}\overline{B}_{-n}\left(
x+k\right)  \label{3.5}
\end{equation}%
and from (\ref{1.3}) and (\ref{2.4}) one obtains the desired formula by
considering the two associated holomorphic functions.
\end{proof}

Explicitly, formula (\ref{3.2}) reads%
\begin{equation}
\varphi _{\overline{B}_{n}}\left( \tau \right) =\left\{ 
\begin{array}{c}
\left\langle \delta ^{\left( -n\right) }\left( x\right) ,\frac{1}{\pi }\frac{%
1}{x-\tau }\right\rangle _{\left[ 0,1\right] }=\left( -1\right) ^{-n+1}\frac{%
\left( -n\right) !}{\pi \tau ^{-n+1}}\text{ \ \ \ \ }if\text{ }n<0 \\ 
\left\langle -1+\delta \left( x\right) ,\frac{1}{\pi }\frac{1}{x-\tau }%
\right\rangle _{\left[ 0,1\right] }=-\frac{1}{\pi \tau }\left[ 1+\tau \log
\left( 1-1/\tau \right) \right] \text{ \ }\ \ \text{\ \ \ }if\text{ }n=0 \\ 
\left\langle -\frac{1}{n!}B_{n}\left( x\right) ,\frac{1}{\pi }\frac{1}{%
x-\tau }\right\rangle _{\left[ 0,1\right] }=-\frac{1}{\pi n!}\left[
B_{n}\left( \tau \right) \log \left( 1-1/\tau \right) +R_{n}\left( \tau
\right) \right] \text{ \ }\ \ \text{\ }if\text{ }n>0%
\end{array}%
\right. \text{,}  \label{3.3}
\end{equation}%
where $R_{n-1}$ is the polynomial of degree $n-1$ such that $B_{n}\left(
\tau \right) \log \left( 1-1/\tau \right) +R_{n}\left( \tau \right) \in
O^{1}\left( \overline{\mathbb{C}}\backslash \left[ 0,1\right] \right) $.
Here are the first six polynomials%
\begin{equation*}
R_{1}\left( \tau \right) =1\text{, \ }R_{2}\left( \tau \right) =\tau -\frac{1%
}{2}\text{, \ }R_{3}\left( \tau \right) =\tau ^{2}-\tau +\frac{1}{12}\text{,}
\end{equation*}%
\begin{equation*}
R_{4}\left( \tau \right) =\tau ^{3}-\frac{3}{2}\tau ^{2}+\frac{1}{3}\tau +%
\frac{1}{12}\text{, \ \ }R_{5}\left( \tau \right) =\tau ^{4}-2\tau ^{3}+%
\frac{3}{4}\tau ^{2}+\frac{1}{4}\tau -\frac{13}{360}\text{,}
\end{equation*}%
\begin{equation}
R_{6}\left( \tau \right) =\tau ^{5}-\frac{5}{2}\tau ^{4}+\frac{4}{3}\tau
^{3}+\frac{1}{2}\tau ^{2}-\frac{13}{60}\tau -\frac{7}{120}\text{.}
\label{3.4}
\end{equation}

\textbf{Remark}. One could object that we should have used the simplest
Appell sequence $\left\{ x^{n}\right\} _{n\in \mathbb{N}}$ and their
associated hyperfunctions on $\left[ 0,1\right] $ to write the generalized
Lipschitz summation formula. But this wold trivially give zero, since%
\begin{equation}
x^{n}=\frac{1}{n+1}\left[ B_{n+1}\left( x+1\right) -B_{n+1}\left( x\right) %
\right]  \label{3.6}
\end{equation}%
and the function $B_{n+1}\left( x+1\right) -B_{n+1}\left( x\right) $ clearly
belongs to the kernel of $\sum_{\mathbb{Z}}$. Instead, other less trivial
Appell sequences can be used to provide useful generalizations of the
construction we proposed, as will be illustrated in the subsequent sections.

\vspace{2mm}

\section{The general case: sequences of Appell polynomials and hyperfunctions%
}

\subsection{Appell sequences}

\vspace{1mm}

By analogy with the theory developed in the previous section, we study a
more general class of polynomials of Bernoulli type, with the aim of
constructing associated families of hyperfunctions. The main request is that
they belong to the class of Appell polynomials, i.e. polynomials satisfying
the property%
\begin{equation}
\frac{d}{dx}A_{n}\left( x\right) =n\text{ }A_{n-1}\left( x\right) \text{,}
\label{Ap1a}
\end{equation}%
with the normalization%
\begin{equation}
A_{0}\left( x\right) =const\text{.}  \label{Ap1b}
\end{equation}

\begin{lemma}
The Fourier series expansion of the polynomials (\ref{Ap1a}), for $0<x<1$
and $n\geq 1$ has the form%
\begin{equation}
A_{n}\left( x\right) =\sum_{k=-\infty }^{\infty }c_{k}\left( n\right)
e^{2\pi ikx}\text{,}  \label{Ap2}
\end{equation}%
with%
\begin{equation}
c_{k}=\int_{0}^{1}A_{n}\left( t\right) e^{-2\pi ikt}dt\text{.}  \label{Ap2b}
\end{equation}%
We get%
\begin{equation}
A_{n}\left( x\right) =-n!\sum_{k=1}^{\infty }\left[ \sum_{j=1}^{n}\frac{1}{j!%
}\frac{\varphi _{j}}{\left( 2\pi ik\right) ^{n+1-j}}e^{2\pi
ikx}+\sum_{j=1}^{n}\left( -1\right) ^{n+1-j}\frac{1}{j!}\frac{\varphi _{j}}{%
\left( 2\pi ik\right) ^{n+1-j}}e^{-2\pi ikx}\right] +\text{ }c_{0}\text{, \
\ }  \label{Ap3}
\end{equation}%
with $0<x<1$, and $\varphi _{j}=A_{j}\left( 1\right) -A_{j}\left( 0\right) $%
, $\ j=1,..,n$. The standard Bernoulli polynomials correspond to the case $%
c_{0}=0,$ $\varphi _{1}=1$, $\varphi _{j}=0$, $\ j=2,3,...$
\end{lemma}

\begin{proof}
It is a direct consequence of the conditions (\ref{Ap1a})--(\ref{Ap2b}) and
of the formula of integration by parts. %\ref{Ap1b}).
\end{proof}

\vspace{2mm}

If $\varphi _{j}=0$ for $j$ even or $j~$odd, then formula (\ref{Ap3}) can be
written in a more compact form. The corresponding polynomial sequences will
be denoted by $\left\{ p_{n}\left( x\right) \right\} _{n\in \mathbb{N}}$ and 
$\left\{ q_{n}\left( x\right) \right\} _{n\in \mathbb{N}}$, respectively.

\vspace{2mm}

a) If $\varphi _{j}=0$ for $j$ even, we define%
\begin{equation}
p_{n}\left( x\right) =:-n!\sum_{k=1}^{\infty }\sum_{\substack{ j=1  \\ j%
\text{ odd}}}^{n}\frac{1}{j!}\frac{\varphi _{j}}{\left( 2\pi ik\right)
^{n+1-j}}\left[ e^{2\pi ikx}+\left( -1\right) ^{n}e^{-2\pi ikx}\right] +%
\text{ }c_{0}\text{, \ \ }0<x<1\text{.}  \label{Ap4}
\end{equation}

b) If $\varphi _{j}=0$ for $j$ odd, we introduce%
\begin{equation}
q_{n}\left( x\right) =:-n!\sum_{k=1}^{\infty }\sum_{\substack{ j=1  \\ j%
\text{ even}}}^{n}\frac{1}{j!}\frac{\varphi _{j}}{\left( 2\pi ik\right)
^{n+1-j}}\left[ e^{2\pi ikx}+\left( -1\right) ^{n+1}e^{-2\pi ikx}\right] +%
\text{ }c_{0}\text{, \ \ }0<x<1\text{.}  \label{Ap5}
\end{equation}%
\textbf{Remark}. The previous conditions on $\varphi _{j}$ are not
particularly restrictive. Indeed, one can easily construct infinitely many
polynomial sequences possessing the prescribed parity properties. Their
generating function has the general form%
\begin{equation}
\frac{te^{xt}}{\left( e^{t}-1\right) g\left( t\right) }=\sum_{n=0}^{\infty
}A_{n}\left( x\right) \frac{t^{n}}{n!}\text{,}  \label{Ap6}
\end{equation}%
where $g\left( t\right) $ is any real analytic function such that $%
\lim_{t\rightarrow 0}g\left( t\right) =1$ and satisfying the parity
condition $g\left( t\right) =g\left( -t\right) $ for the case a) and $%
g\left( t\right) =-g\left( -t\right) $ for the case b). \ 

\vspace{4mm}

\textbf{Examples}

\vspace{4mm}

a) Take%
\begin{equation}
\frac{t^{2}e^{xt}}{\left( e^{t}-1\right) \sin t}=\sum_{n=0}^{\infty
}a_{n}\left( x\right) \frac{t^{n}}{n!}  \label{Ap7}
\end{equation}%
We easily deduce

\begin{equation*}
a_{0}\left( x\right) =1\text{, \ }a_{1}\left( x\right) =x-\frac{1}{2}\text{,
\ }a_{2}\left( x\right) =x^{2}-x+\frac{1}{2}\text{,}
\end{equation*}%
\begin{equation*}
a_{3}\left( x\right) =x^{3}-\frac{3}{2}x^{2}+\frac{3}{2}x-\frac{1}{2}\text{,}
\end{equation*}%
\begin{equation}
a_{4}\left( x\right) =x^{4}-2x^{3}+3x^{2}-2x+\frac{23}{30}\text{,...}
\label{Ap8}
\end{equation}

\vspace{2mm}

b) Consider the generating function 
\begin{equation}
\frac{te^{xt}}{\left( e^{t}-1\right) \cos t}=\sum_{n=0}^{\infty }b_{n}\left(
x\right) \frac{t^{n}}{n!}  \label{Ap9}
\end{equation}%
We immediately get

\begin{equation*}
b_{0}\left( x\right) =1\text{, \ }b_{1}\left( x\right) =x-\frac{1}{2}\text{,
\ }b_{2}\left( x\right) =x^{2}-x+\frac{7}{6}\text{,}
\end{equation*}%
\begin{equation*}
b_{3}\left( x\right) =x^{3}-\frac{3}{2}x^{2}+\frac{7}{2}x-\frac{3}{2}\text{,}
\end{equation*}%
\begin{equation}
b_{4}\left( x\right) =x^{4}-2x^{3}+7x^{2}-6x+\frac{179}{30}\text{,...}
\label{Ap10}
\end{equation}

\subsection{Extended delta rational functions}

\vspace{1mm}

Let $n\in \mathbb{Z}$, $q\in \mathbb{D}$. \ A straightforward generalization
of $\delta _{n}\left( q\right) $ adapted to the chosen Appell polynomials is
provided by the extended delta rational function $\Delta _{n}\left( q\right) 
$.

\begin{definition}
The extended delta rational function, for any $n\in \mathbb{Z}$ and $q\in 
\mathbb{D}$, is defined by%
\begin{equation}
\Delta _{n}\left( q\right) =\left\{ 
\begin{array}{c}
\sum_{k=1}^{\infty }k^{n}q^{k}\,\ \ \ \ \ \ \ \ \ \ \ \ n>0 \\ 
\\ 
\sum_{k=1}^{\infty }a_{k}\left( n\right) q^{k}\ \ \ \ \ \ \ n\leq 0%
\end{array}%
\right.  \label{5.1}
\end{equation}%
where 
\begin{equation}
a_{k}\left( n\right) =\sum_{\substack{ j  \\ j\text{ even or odd}}}^{n}\frac{%
1}{j!}\frac{\varphi _{j}}{k^{n+1-j}}\text{.}  \label{5.2}
\end{equation}
\end{definition}

In eq. (\ref{5.2}), the summation should be understood either over the even
values of $j$ or over the odd ones, depending on the choice of the
polynomials (\ref{Ap4}) or (\ref{Ap5}), respectively. The above definition
is motivated by the following result, which provides an extension of the
construction proposed in Section 3. Our aim is to obtain the hyperfunctional
equations associated to the proposed Appell polynomials. As a consequence of
the Fourier expansion (\ref{Ap3}) and of relations (\ref{Ap4})--(\ref{Ap5})
we get the relation between extended delta rational functions and Appell
sequences.

\vspace{3mm}

\begin{lemma}
The following inversion equations, generalizing relation (\ref{1.3}), hold%
\begin{equation}
\Delta _{n}\left( q\right) +\left( -1\right) ^{n}\Delta _{n}\left(
q^{-1}\right) =\left\{ 
\begin{array}{c}
0\qquad \qquad \qquad \qquad \qquad \ \ \ if\text{ }n\geq 1 \\ 
-\frac{\left( 2\pi i\right) ^{-n}}{\left( -n\right) !}P_{-n}\left( \frac{%
\log q}{2\pi i}\right) \text{ \qquad }if\text{ }n\leq 0%
\end{array}%
\right. \text{,}  \label{5.3}
\end{equation}%
and%
\begin{equation}
\Delta _{n}\left( q\right) +\left( -1\right) ^{n+1}\Delta _{n}\left(
q^{-1}\right) =\left\{ 
\begin{array}{c}
0\qquad \qquad \qquad \qquad \qquad \ \ \ if\text{ }n\geq 1 \\ 
-\frac{\left( 2\pi i\right) ^{-n}}{\left( -n\right) !}Q_{-n}\left( \frac{%
\log q}{2\pi i}\right) \text{ \qquad }if\text{ }n\leq 0%
\end{array}%
\right. \text{,}  \label{5.4}
\end{equation}%
where $P_{n}$ and $Q_{n}$ are respectively the Appell polynomials (\ref{Ap4}%
) and (\ref{Ap5}), to which the constant $c_{0}$ has been subtracted. The
inversion relations hold for all $q\neq 1$ if $n\geq 0$, whereas for $n\leq
0 $ $q$ can be taken in $%
%TCIMACRO{\U{2102} }%
%BeginExpansion
\mathbb{C}
%EndExpansion
\backslash \left[ 0,+\infty \right] $.
\end{lemma}

By analogy with formulae (\ref{2.1}), the associated periodic functions and
distributions are defined as

\begin{equation}
\widetilde{P}_{n}\left( x\right) =\left\{ 
\begin{array}{c}
-\frac{P_{n}\left( x-\left[ x\right] \right) }{n!}\qquad if\text{ }n\geq 1
\\ 
-1+\delta _{\mathbb{T}}\left( x\right) \text{ \ \ \ }if\text{ }n=0 \\ 
\delta _{\mathbb{T}}^{\left( -n\right) }\left( x\right) \qquad if\text{ }n<0%
\end{array}%
\right. \text{.}  \label{5.5}
\end{equation}%
with an analogous definition for the case of polynomials $Q_{n}\left(
x\right) $. For all $n\in \mathbb{Z}$ we have the hyperfunctional equations

\begin{equation}
\Delta _{n}\left( q^{+}\right) +\left( -1\right) ^{n}\Delta _{n}\left(
\left( q^{-}\right) ^{-1}\right) =\left( 2\pi i\right) ^{-n}\widetilde{P}%
_{-n}\left( x\right) \text{,}  \label{5.6}
\end{equation}%
and

\begin{equation}
\Delta _{n}\left( q^{+}\right) +\left( -1\right) ^{n+1}\Delta _{n}\left(
\left( q^{-}\right) ^{-1}\right) =\left( 2\pi i\right) ^{-n}\widetilde{Q}%
_{-n}\left( x\right) \text{.}  \label{5.7}
\end{equation}%
The proof of these relations is again a direct consequence of the previous
definitions. Also, by denoting with $\overline{P}_{n}$ the restriction to $%
\left[ 0,1\right] $ of $\widetilde{P_{n}}$ and with $\varphi _{\overline{P}%
_{n}}$ the function in $O^{1}\left( \overline{\mathbb{C}}\backslash \left[
0,1\right] \right) $ which represents the hyperfunction $\overline{P}_{n}$,
namely

\begin{equation}
\varphi _{\overline{P}_{n}}\left( \tau \right) =\left\langle \overline{P}%
_{n},\frac{1}{\pi }\frac{1}{x-\tau }\right\rangle _{\left[ 0,1\right] }\text{%
,}
\end{equation}%
we find, for any choice of the sequence $\left\{ P_{n}\left( x\right)
\right\} _{n\in \mathbb{N}}$ and $\left\{ Q_{n}\left( x\right) \right\}
_{n\in \mathbb{N}}$ our main result, i.e. the \textit{generalized Lipschitz
formula}

\begin{equation}
\sum_{k\in \mathbb{Z}}\varphi _{\overline{P}_{n}}\left( \tau +k\right)
=2i\left( 2\pi i\right) ^{-n}\left\{ 
\begin{array}{c}
\Delta _{-n}\left( q\right) \text{ \ \ \ \ \ \ \ \ \ \ \ \ \ \ \ \ \ \ }if%
\text{ }\left\vert q\right\vert <1,\,\ i.e.\text{ }\Im \tau >0 \\ 
\\ 
\left( -1\right) ^{n-1}\Delta _{-n}\left( q^{-1}\right) \text{ \ \ }if\text{ 
}\left\vert q\right\vert >1,\,\ i.e.\text{ }\Im \tau <0%
\end{array}%
\right. \text{,}
\end{equation}%
and the corresponding formula for $\overline{Q}_{n}$ (again the usual
Lipschitz formula corresponds to $n\leq -1$). The explicit computation of $%
\varphi _{\overline{P}_{n}}\left( \tau \right) $ and $\varphi _{\overline{Q}%
_{n}}\left( \tau \right) $ is completely analogous to the proposed
construction for the Bernoulli polynomials and we will not repeat it here.

\vspace{2mm}

\section{A connection with the Lazard formal group}

\vspace{1mm}

We describe here briefly an interesting connection between formal groups,
hyperfunctions, and the so--called universal Bernoulli polynomials.

Given a commutative ring with identity $R$, we will denote by $R\left[ x_{1},%
\text{ }x_{2},..\right] $ the ring of formal power series in $x_{1}$, $x_{2}$%
, ... with coefficients in $R$. Following \cite{Haze}, \cite{BMN}, we recall
that a commutative one--dimensional formal group law over $R$ is a
two--variable formal power series $\Phi \left( x,y\right) \in R\left[ x,y%
\right] $ such that%
\begin{equation*}
1)\qquad \Phi \left( x,0\right) =\Phi \left( 0,x\right) =x
\end{equation*}%
and%
\begin{equation*}
2)\qquad \Phi \left( \Phi \left( x,y\right) ,z\right) =\Phi \left( x,\Phi
\left( y,z\right) \right) \text{.}
\end{equation*}

When $\Phi \left( x,y\right) =\Phi \left( y,x\right) $, the formal group is
said to be commutative. The existence of an inverse formal series $\varphi
\left( x\right) $ $\in R\left[ x\right] $ such that $\Phi \left( x,\varphi
\left( x\right) \right) =0$ follows from the previous definition.

Let us consider the polynomial ring $\mathbb{Q}\left[ c_{1},c_{2},...\right] 
$ and the formal power series

\begin{equation}
F\left( s\right) =s+c_{1}\frac{s^{2}}{2}+c_{2}\frac{s^{3}}{3}+...
\label{Ap11}
\end{equation}%
Let $G\left( t\right) $ be the associated inverse series

\begin{equation}
G\left( t\right) =t-c_{1}\frac{t^{2}}{2}+\left( 3c_{1}^{2}-2c_{2}\right) 
\frac{t^{3}}{6}+...  \label{Ap12}
\end{equation}%
so that \textit{$F\left( G\left( t\right) \right) =t$}. The series (\ref%
{Ap11}) and (\ref{Ap12}) are called formal group logarithm and formal group
exponential, respectively. The formal group law related to these series is
provided by 
\begin{equation*}
\Phi \left( s_{1},s_{2}\right) =G\left( F\left( s_{1}\right) +F\left(
s_{2}\right) \right)
\end{equation*}%
$\,$ and it represents the so called \textit{Lazard's Universal Formal Group 
}\cite{Haze}. It is defined over the Lazard ring $L$, i.e. the subring of $%
\mathbb{Q}\left[ c_{1},c_{2},...\right] $ generated by the coefficients of
the power series $G\left( F\left( s_{1}\right) +F\left( s_{2}\right) \right) 
$.

In \cite{Tempesta1}, the \textit{universal Bernoulli polynomials\ $%
B_{n}^{G}\left( x,c_{1},...,c_{k},...\right) \equiv B_{k}^{G}\left( x\right) 
$ related to the formal group exponential }$G$ have been introduced. They
are defined by%
\begin{equation}
\frac{t}{G\left( t\right) }e^{xt}=\sum_{n\geq 0}B_{n}^{G}\left( x\right) 
\frac{t^{k}}{n!}\text{,}\qquad \qquad x\in \mathbb{R}\text{.}  \label{Ap13}
\end{equation}%
The corresponding numbers by construction coincide with the universal
Bernoulli numbers discovered in \cite{Clarke}, and generated by%
\begin{equation}
\frac{t}{G\left( t\right) }=\sum_{n\geq 0}\widehat{B_{n}}\frac{t^{k}}{n!}%
,\qquad \qquad x\in \mathbb{R}.  \label{Ap14}
\end{equation}

Observe that when $a=1$, $c_{i}=\left( -1\right) ^{i}$, then $F\left(
s\right) =\log \left( 1+s\right) \,$, $G\left( t\right) =e^{t}-1$, and the
universal Bernoulli polynomials and numbers reduce to the standard ones.
Many other examples of Bernoulli--type polynomials considered in the
literature are obtaining by specializing the rational coefficients $c_{i}$.
The reasons to consider such a generalization of Bernoulli polynomials are
manifold. First, to any choice of the coefficients $c_{1},c_{2},...$ it
corresponds a sequence of polynomials of Appell type sharing with the
standard Bernoulli polynomials many algebraic and combinatorial properties 
\cite{Tempesta2}. In particular, the associated numbers satisfy universal
Clausen--von Staudt and Kummer congruences, as shown in the Appendix. Also,
in the same way as the Riemann zeta function is associated with the
Bernoulli polynomials, it is possible to relate the polynomials (\ref{Ap13})
with a large class of absolutely convergent L--series. Their values at
negative integers correspond to the generalized Bernoulli numbers, and they
possess as well several interesting number--theoretical properties \cite%
{Tempesta1}. Observe that the polynomials (\ref{Ap6}) belong to the class (%
\ref{Ap13}), i.e. the coefficient $c_{i}$ are all rational, if the functions 
$g\left( t\right) $ verify the further condition $g^{\left( n\right) }\left(
0\right) $ $\in \mathbb{Q}$ for any $n$. This is the case, for instance, for
the proposed examples of the polynomials (\ref{Ap7}) and (\ref{Ap9}). As an
immediate consequence, the numbers generated by%
\begin{equation}
\frac{t}{\left( e^{t}-1\right) g\left( t\right) }=\sum_{n=0}^{\infty }A_{n}%
\frac{t^{n}}{n!}\text{,}
\end{equation}%
with the prescribed conditions on $g\left( t\right) $ do satisfy by
construction the Clausen--von Staudt and Kummer congruences and many others.
These considerations enable us to associate one--variable formal groups with
suitable classes of hyperfunctions, as a consequence of the previous
construction.

\vspace{2mm}

\section{Appendix: congruences}

\vspace{2mm}

The Clausen--von Staudt congruence \cite{IR}, one of the most beautiful of
mathematics, states that%
\begin{equation}
B_{n}+\sum_{p\mid n}\frac{1}{p}\in \mathbb{Z}\text{,}  \label{CvS}
\end{equation}%
where $B_{n}$ denotes the $n$--th Bernoulli number. This proves the strict
link between Bernoulli numbers and prime numbers. Many generalizations of
this result have been obtained in the literature in the last decades. In an
attempt to clarify the deep connection between these congruences and
algebraic topology, in \cite{Clarke} Clarke proposed the notion of universal
Bernoulli numbers $\widehat{B_{n}}$, defined as (\ref{Ap14}), and proved the
remarkable universal von Staudt congruence.

If $n$ is even, we have%
\begin{equation}
\widehat{B_{n}}\equiv -\sum_{\overset{p-1\mid n}{p\text{ prime}}}\frac{%
c_{p-1}^{n/(p-1)}}{p}\qquad mod\quad \mathbb{Z}\left[ c_{1},c_{2},...\right] 
\text{;}  \label{US1}
\end{equation}%
If $n$ is odd and greater than 1, we have%
\begin{equation}
\widehat{B_{n}}\equiv \frac{c_{1}^{n}+c_{1}^{n-3}c_{3}}{2}\qquad mod\quad 
\mathbb{Z}\left[ c_{1},c_{2},...\right] \text{.}  \label{US2}
\end{equation}%
Like the classical ones, the universal Bernoulli numbers as well play an
important role in several branches of mathematics, in particular in complex
cobordism theory (see e.g. \cite{BCRS}, and \cite{Ray}), where the
coefficients $c_{n}$ are identified with the cobordism classes of $\mathbb{C}%
P^{n}$.

Kummer congruences are also relevant in algebraic topology, and in defining
p--adic extensions of zeta functions. We also have an universal Kummer
congruence \cite{Adel}. Suppose that $n\neq 0,1$ (mod $p-1$). Then

\begin{equation}
\frac{\widehat{B}_{n+p-1}}{n+p-1}\equiv \frac{\widehat{B}_{n}}{n}%
c_{p-1}\qquad mod\quad p\mathbb{Z}_{p}\left[ c_{1},c_{2},...\right] \text{.}
\label{UK}
\end{equation}

Other related congruences can be found in \cite{Adel}.


\begin{thebibliography}{99}
\bibitem{Adel} A. Adelberg, Universal higher order Bernoulli numbers and
Kummer and related congruences, J. Number Theory \textbf{84} (2000),
119--135.

\bibitem{BCRS} A. Baker, F. Clarke, N. Ray, L. Schwartz, On the Kummer
congruences and the stable homotopy of BU, Trans. Amer. Math. Soc. \textbf{%
316} (1989), 385--432.

\bibitem{Bruggeman} R. Bruggeman, Automorphic forms, hyperfunction
cohomology, and period functions, J. reine angew. Math. \textbf{492} (1997),
1--39.

\bibitem{BMN} V. M. Bukhstaber, A. S. Mishchenko and S. P. Novikov, Formal
groups and their role in the apparatus of algebraic topology, \textit{Uspehi
Mat. Nauk}, \textbf{26}(2(158)): 63--90, 1971.

\bibitem{Clarke} F. Clarke, The universal Von Staudt theorems, Trans. Amer.
Math. Soc. \textbf{315} (1989), 591--603 .

\bibitem{Haze} M. Hazewinkel, \textit{Formal Groups and Applications},
Academic Press, New York, 1978.

\bibitem{Horm} L. H\"{o}rmander, \textit{The Analysis of Partial
Differential Operators I}, second edition, Springer--Verlag, 1990.

\bibitem{IR} K. Ireland and M. Rosen, \textit{A Classical Introduction to
Modern Number Theory}, Springer--Verlag, 1982.

\bibitem{Marmi} S. Marmi, P. Moussa, J.--C. Yoccoz, Complex Brjuno
functions, J. Amer. Math. Soc. \textbf{14}, n. 4 (2001), 783--841.

\bibitem{Kaneko} A. Kaneko, \textit{Introduction to Hyperfunctions}, Kluwer
Academic Publisher (1988).

\bibitem{Lewin} L. Lewin (ed.) \textit{Structural properties of
polylogarithms}, Math. Surveys Monogr. \textbf{37} AMS (1991).

\bibitem{PD} P. C. Pasles and W. De Azevedo Pribitkin, A generalization of
the Lipschitz summation formula and some applications, Proc. Am. Math. Soc. 
\textbf{129} (2001), 3177--3184.

\bibitem{Pham} A. Pham (ed.), \textit{Hyperfunctions and Theoretical Physics}%
, Lecture Notes in Mathematics, 449, Springer--Verlag, 1975.

\bibitem{Ray} N. Ray, Stirling and Bernoulli numbers for complex oriented
homology theory, in Algebraic Topology, Lecture Notes in Math. \textbf{1370}%
, 362--373, G. Carlsson, R. L. Cohen, H. R. Miller and D. C. Ravenel (Eds.),
Springer--Verlag, 1986.

\bibitem{SKK} M. Sato, T. Kawai, M. Kashiwara, Microfunctions and
pseudo--differential operators. \textit{Hyperfunctions and
pseudo--differential equations}, Lecture Notes in Math \textbf{287},
265--529, Springer, Berlin, 1973.

\bibitem{Tempesta1} P. Tempesta, Formal Groups, Bernoulli--type polynomials
and L--series, C. R. Math. Acad. Sci. Paris (2007), Ser. I \textbf{345}
(2007), 303--306.

\bibitem{Tempesta2} P. Tempesta, On Appell sequences of polynomials of
Bernoulli and Euler type, J. Math. Anal. Appl. (2007),
doi:10.1016/j.jmaa.2007.07.018 (in press).

\bibitem{WMLI} M. Waldschmidt, P. Moussa, J. M. Luck and C. Itzykson eds., 
\textit{From Number Theory to Physics}, Springer--Verlag, Berlin, 1992.

\bibitem{Zagier} D. Zagier, The dilogarithm function, in \textit{Frontiers
in Number Theory, Physics and Geometry II}, P. Cartier, B. Julia, P. Moussa
and P. Vanhove (eds.) 3--65, Springer, 2007.
\end{thebibliography}
\end{document}